\documentclass[12pt, leqno]{article}

\usepackage[margin=1in]{geometry}
\usepackage{amsmath, amssymb, amsthm, mathtools}
\usepackage{mathrsfs}
\usepackage{enumitem}
\usepackage[colorlinks=true,linkcolor=blue,citecolor=blue,urlcolor=blue]{hyperref}
\usepackage{cite}
\usepackage{xcolor}
\usepackage{authblk}

\newtheorem{theorem}{Theorem}[section]

\newtheorem{proposition}[theorem]{Proposition}
\newtheorem{lemma}[theorem]{Lemma}
\newtheorem{corollary}[theorem]{Corollary}

\theoremstyle{definition}
\newtheorem{definition}[theorem]{Definition}

\newtheorem{remark}[theorem]{Remark}
\newtheorem{example}[theorem]{Example}

\numberwithin{equation}{section}

\newcommand{\R}{\mathbb R}

\newcommand{\jet}[2]{J^{(#1)}_{#2}}
\newcommand{\norm}[1]{\left\lVert #1 \right\rVert}

\DeclareMathOperator{\supp}{supp}

\newcommand{\set}[1]{\left\{#1\right\}}

\newcommand{\eqindent}{\displayindent0pt\displaywidth\textwidth}

\renewcommand{\epsilon}{\varepsilon}
\newcommand{\eps}{\epsilon}

\title{Smooth solutions to systems of linear inequalities on $\mathbb{R}$}
\author[1]{Fushuai Jiang}
\author[2]{Garving K. Luli}
\affil[1]{Department of Mathematics, City University of Hong Kong}
\affil[2]{Department of Mathematics, University of California - Davis}
\date{}
\begin{document}

\maketitle

\begin{abstract}
Fix an integer $m\geq 0$. We study the one-dimensional problem of deciding when a system of linear inequalities
\begin{equation*}
    \sum_{j=1}^M A_{ij}(x)F_j(x)\leq f_i(x),\qquad i=1,\ldots,N,
\end{equation*}
with fixed semialgebraic coefficients $A_{ij}:\mathbb R\to\mathbb R$ and given functions
\begin{equation*}
    f_i\in C^\infty(\mathbb R),\qquad i=1,\ldots,N,
\end{equation*}
admits a solution
\begin{equation*}
    F=(F_1,\ldots,F_M)\in C^m(\mathbb R,\mathbb R^M).
\end{equation*}
The analogous problem for systems of linear equations admits a finite linear differential criterion, whereas for systems of inequalities such a criterion fails in general in dimensions at least two, already for continuous solutions. This paper gives a complete answer in one dimension for a system of linear inequalities. For $m=0,1,2$, the existence of a $C^m$ solution is characterized by a finite differential criterion. In contrast, for every $m\geq 3$, no finite criterion of this local differential type exists in general, even for a fixed constant-coefficient system. Thus, $m=2$ is the sharp regularity threshold for finite differential characterization of the one-dimensional inequality problem.
\end{abstract}

\section{Introduction}

Let $m\geq 0$ be an integer. Let $M,N\geq 1$, and for $1\leq i\leq N$ and $1\leq j\leq M$, let
\[
A_{ij}:\mathbb R\to\mathbb R
\]
be fixed semialgebraic functions, meaning that their graphs are finite unions of sets defined by polynomial equalities and inequalities (see~\cite{BochnakCosteRoy1998}). Given
\[
f=(f_1,\ldots,f_N)\in C^\infty(\mathbb R,\mathbb R^N),
\]
we ask whether there exists
\[
F=(F_1,\ldots,F_M)\in C^m(\mathbb R,\mathbb R^M)
\]
such that
\begin{equation}
\sum_{j=1}^M A_{ij}(x)F_j(x)\leq f_i(x),
\qquad x\in\mathbb R,\quad i=1,\ldots,N.
\label{1.1}
\end{equation}

The problem is posed for arbitrary regularity $m\geq 0$.  In this paper, we show in Theorem~\ref{thm:main-theorem} that for $m=0,1,2$, the existence of a $C^m$ solution can be characterized by finitely many local linear-differential sign conditions on the data. We also show in Theorem~\ref{thm:C3 sharp} that, starting at $m=3$, no criterion of this finite local differential type exists in general, even for a fixed constant-coefficient system.

In \cite{FeffermanLuli2021Equations,FeffermanLuli2021Generators}, C.~Fefferman and the second author proved that the corresponding problem for systems of linear equations has a finite differential criterion. For fixed semialgebraic coefficients $A_{ij}$, the solvability of
\begin{equation}
\sum_{j=1}^M A_{ij}(x)F_j(x)=f_i(x),
\qquad i=1,\ldots,N,
\label{1.2}
\end{equation}
can be characterized by finitely many linear differential equations in $f$, with semialgebraic coefficients. More precisely, they proved the following:

\begin{theorem}[\hspace{1sp}\cite{FeffermanLuli2021Equations,FeffermanLuli2021Generators}]
Fix $m\geq 0$, and let $(A_{ij}(x))_{1\leq i\leq N,\ 1\leq j\leq M}$ be a matrix of semialgebraic functions on $\mathbb R^n$. Then there exist $\bar m\geq m$ and a finite list of linear partial differential operators
\[
L_\nu f(x)
=
\sum_{i=1}^N
\sum_{|\alpha|\leq \bar m}
a_{\nu i\alpha}(x)\partial^\alpha f_i(x),
\]
with semialgebraic coefficient functions $a_{\nu i\alpha}$, such that \eqref{1.2} admits a solution
\[
F\in C^m(\mathbb R^n,\mathbb R^M)
\]
if and only if
\[
L_\nu f\equiv 0
\text{\quad for every $\nu$.}
\]
\end{theorem}

Let us review an example of Epstein and Hochster \cite{EpsteinHochster2018} to illustrate the content of this theorem concretely. Consider the single linear equation
\begin{equation}\label{eq:Epstein Hochster example}
    x^2F_1+y^2F_2+xyz^2F_3=f(x,y,z).
\end{equation}

There exist continuous $F_1,F_2,F_3$ (no differentiability required) satisfying \eqref{eq:Epstein Hochster example} if and only if
\begin{equation}
\left[ 
\begin{array}{l}
f\left( x,y,z\right) =\frac{\partial f}{\partial x}\left( x,y,z\right) =%
\frac{\partial f}{\partial y}\left( x,y,z\right) =0 \\ 
\text{and} \\ 
\frac{\partial ^{2}f}{\partial x\partial y}\left( x,y,z\right) =\frac{%
\partial ^{3}f}{\partial x\partial y\partial z}\left( x,y,z\right) =0%
\end{array}%
\right. 
\begin{array}{l}
\text{for }x=y=0,z\in \mathbb{R} \\ 
\\ 
\text{at }x=y=z=0\text{.}%
\end{array}
\label{intro3.3}
\end{equation}

For the history and recent progress on this problem, see \cite{BierstoneCampesatoMilman2021,Fefferman2006WhitneyCm,FeffermanIsraelLuli2016Selection,FeffermanLuli2022Plane,JiangLuliONeill2022Shape,LuliONeill2023Inequalities,JiangLuliONeill2022Smooth} and the references therein.

In \cite{LuliONeill2023Inequalities}, the authors ask for an inequality analogue of this theorem. They show that such a criterion fails in general in dimensions $n\geq 2$, already for continuous solutions. The corresponding problem for one dimension was left open. The case $n=1$ is fundamentally different. The real line has only two one-sided directions at a point, and one-variable semialgebraic functions have finite Puiseux asymptotics. These two features permit a finite analysis of the behavior at the exceptional points of the coefficient matrix.

A second ingredient is the one-dimensional Riesz interpolation theorem of Nagel and Rudin~\cite{NagelRudin1974}. For each
\[
m\in\{0,1,2\},
\]
the ordered space $C^m(\mathbb R)$ has the Riesz decomposition property, which is equivalent to the weak-lattice interpolation property: if $L_1,L_2,U_1,U_2\in C^m(\mathbb R)$
satisfy
\[
\max\{L_1,L_2\} \leq \min\{U_1,U_2\},
\]
% \[
% L_i\leq U_j
% \qquad
% (i,j=1,2),
% \]
then there exists $h\in C^m(\mathbb R)$
such that
% \[
% L_i\leq h\leq U_j
% \qquad
% (i,j=1,2),
% \]
% or equivalently,
\[
\max\{L_1,L_2\}\leq h\leq \min\{U_1,U_2\}.
\]
This is exactly the scalar insertion step needed in the Fourier--Motzkin elimination. Thus the regular-interval construction works directly, and at the correct regularity, for each of the three cases $m=0,1,2$.

The regularity $m=2$ is the endpoint of this phenomenon. Nagel and Rudin \cite{NagelRudin1974} also exhibit $C^\infty$ data for which the corresponding $B^3$ Riesz decomposition problem has no solution, where
\[
B^3(\mathbb R)
=
\left\{
g\in C^2(\mathbb R):
g'''(x)\text{ exists for every }x\in\mathbb R
\text{ and }g'''\text{ is locally bounded}
\right\}.
\]
In particular,
\[
C^3(\mathbb R)\subset B^3(\mathbb R).
\]
Their example yields the obstruction used below for $m=3$, and therefore for every $m\geq 3$.

We now state the main result. For $x\in\mathbb R$, write
\begin{equation}
K_f(x)
=
\{v\in\mathbb R^M:A(x)v\leq f(x)\}.
\label{1.3}
\end{equation}

\begin{theorem}[Main theorem]\label{thm:main-theorem}
Let $M,N\geq 1$, and let
\[
A_{ij}:\mathbb R\to\mathbb R,
\qquad
1\leq i\leq N,\quad 1\leq j\leq M,
\]
be semialgebraic functions.

For each $m\in\{0,1,2\}$, there exist an integer $\bar m\geq m$ and a finite semialgebraic partition $\mathcal C$ of $\mathbb R$ with the following property. For every cell $C\in\mathcal C$, there are finitely many linear ordinary differential expressions
\begin{equation}
L_{C,\nu}f(x)
=
\sum_{i=1}^N
\sum_{r=0}^{\bar m}
a_{C,\nu ir}(x)f_i^{(r)}(x),
\qquad
\nu=1,\ldots,K_C,
\label{1.4}
\end{equation}
whose coefficients are semialgebraic on $C$, and a finite Boolean formula $\mathcal S_C$ in the sign conditions
\[
L_{C,\nu}f>0,
\qquad
L_{C,\nu}f=0,
\qquad
L_{C,\nu}f<0,
\]
such that, for every $f\in C^\infty(\mathbb R,\mathbb R^N)$, the following are equivalent:

\begin{enumerate}
\item[(i)] There exists
\[
F\in C^m(\mathbb R,\mathbb R^M)
\]
satisfying
\[
A(x)F(x)\leq f(x)
\qquad
(x\in\mathbb R).
\]

\item[(ii)] For every cell $C\in\mathcal C$ and every $x\in C$,
\[
\mathcal S_C
\bigl(
L_{C,1}f(x),\ldots,L_{C,K_C}f(x)
\bigr)
\]
holds.
\end{enumerate}

The partition, the coefficient functions, the Boolean formulas, and the integer $\bar m$ depend only on $A$, $M$, $N$, and $m$.

In contrast, for every $m\geq 3$, there exists a fixed constant-coefficient system for which no finite criterion depending pointwise on a fixed finite jet of the data can characterize the existence of a $C^m$ solution. In particular, no criterion of the form above exists in general for $m\geq 3$.
\end{theorem}

We take a simple example to illustrate the content of Theorem \ref{thm:main-theorem}. Suppose that $f\in C^\infty(\mathbb R)$, and consider the following inequalities for $x\in\mathbb R$:
\begin{equation}
\begin{cases}
x^2\mathbb{I}_{x\geq 0}F(x)\leq f(x)\leq x\mathbb{I}_{x\geq 0}F(x),\\
x\mathbb{I}_{x\leq 0}F(x)\leq f(x)\leq x^2\mathbb{I}_{x\leq 0}F(x),
\end{cases}
\label{intro3a}
\end{equation}
where the unknown $F$ is required to be continuous on $\mathbb R$. One
checks that a continuous solution $F$ exists if and only if $f$ satisfies
\begin{equation}
\begin{cases}
f(0)=0,\\
f'(0)\geq 0.
\end{cases}
\label{intro3}
\end{equation}

Note that the derivative $f'(0)$ enters into the criterion
\eqref{intro3}, even though we are seeking only a continuous solution $F$.

In dimension one, finite local differential characterization holds exactly for
\[
m=0,1,2.
\]

We describe the proof. Formal Fourier--Motzkin elimination applied to the fixed matrix $(A_{ij}(x))$ produces a finite semialgebraic decomposition of the line into regular intervals and exceptional points. On a regular interval, pointwise nonemptiness of $K_f(x)$ (see~\eqref{1.3}) implies the existence of a $C^m$ selection for each $m\in\{0,1,2\}$. Indeed, the Fourier--Motzkin induction reduces the problem to finitely many scalar lower and upper bounds in $C^m$, and the Nagel--Rudin weak-lattice property in $C^m(\mathbb R)$ supplies the required scalar insertion at the same regularity.

The essential issue is compatibility at an exceptional point. Fix an exceptional point $y$, a side $\varepsilon\in\{-,+\}$, and
\[
P\in\mathcal P^m(\mathbb R,\mathbb R^M),
\]
where $\mathcal P^m(\mathbb R,\mathbb R^M)$ denotes the space of $\mathbb R^M$-valued polynomials of degree at most $m$. Define
\begin{equation}
d^\pm_{y,P}(t)
=
\operatorname{dist}_\infty
\bigl(
P(y\pm t),K_f(y\pm t)
\bigr),
\qquad
t>0,
\label{1.5}
\end{equation}
where
\[
\operatorname{dist}_\infty(u,K)
=
\inf_{v\in K}\|u-v\|_\infty,
\qquad
\|w\|_\infty
=
\max_{1\leq j\leq M}|w_j|.
\]
Explicitly,
\[
d^\pm_{y,P}(t)
=
\inf_{v\in K_f(y\pm t)}
\max_{1\leq j\leq M}
|P_j(y\pm t)-v_j|.
\]

For $m\in\{0,1,2\}$, we prove that $P$ is the $m$-jet of a feasible one-sided $C^m$ selection if and only if
\[
P(y)\in K_f(y)
\text{\quad and \quad }
d^\pm_{y,P}(t)=o(t^m)
\qquad
(t\downarrow 0).
\]
For $m=0$, the condition $o(t^0)$ means simply $o(1)$.

Linear-programming duality expresses the distance as the maximum of finitely many positive parts of linear expressions in $P$ and $f(y\pm t)$. One-variable Puiseux asymptotics and Taylor's theorem then turn
\[
d^\pm_{y,P}(t)=o(t^m)
\]
into finitely many homogeneous linear sign conditions on $P$ and a finite jet of $f$. Eliminating the finite-dimensional jet parameter $P$ gives the exceptional-point criterion.

For the converse direction in the endpoint theorem, the distance condition first produces a small semialgebraic correction to a relaxed Taylor system. A semialgebraic error bound gives a pointwise correction inside a sufficiently small box. Weighted Fourier--Motzkin elimination and the weighted version of the Nagel--Rudin scalar insertion theorem in $C^m$ then produce a genuine $C^m$ correction on the punctured interval with enough decay to extend to the endpoint with zero $m$-jet. The argument is carried out uniformly for $m=0,1,2$, with the endpoint order $m$ and the final weighted order chosen strictly larger than $m$.

Finally, on each regular interval we patch an interior feasible $C^m$ selection to the one-sided endpoint selections by convex combinations. At every exceptional point, the same polynomial $m$-jet is used on the two sides, so the pieces join to a global $C^m$ solution.

The paper is organized as follows. Section \ref{sec:2} treats regular intervals and $C^m$ scalar insertion for $m=0,1,2$. Section \ref{sec:3} develops the one-sided semialgebraic finite-jet principle in the form needed for these three regularities. Section \ref{sec:4} proves the one-sided $C^m$ endpoint criterion and the finite exceptional-point conditions. Section \ref{sec:5} assembles these results and proves the positive part of Theorem \ref{thm:main-theorem} for $m=0,1,2$. Section \ref{sec:6} proves the sharpness of the regularity threshold by showing that finite local differential characterization fails for every $m\geq 3$.

\paragraph{Acknowledgment.}
The first author is supported by a grant from the City University of Hong Kong (Project No.\ 7200844). The second author is supported in part by NSF grant 2247429, the Chancellor's Fellowship at UC Davis, and the Simons Gift Fund.

\section{Regular intervals and \texorpdfstring{$C^m$}{Cm} selection} \label{sec:2}

Throughout this section, fix
\[
m\in\{0,1,2\}.
\]
We isolate the part of the proof that takes place away from the finitely many exceptional points of the fixed semialgebraic coefficient matrix. The arguments in this section are uniform in $m$. The only regularity input is the one-dimensional Riesz decomposition, or equivalently weak-lattice interpolation, property of $C^m(\mathbb R)$ for $m=0,1,2$.

\subsection{Value polyhedra and Fourier--Motzkin elimination}

For the system
\[
A(x)F(x)\leq f(x),
\]
define
\[
K_f(x)=\{v\in\mathbb R^M:A(x)v\leq f(x)\}.
\]
Thus $F$ solves the system on a set $E\subset\mathbb R$ if and only if
\[
F(x)\in K_f(x)
\qquad
(x\in E).
\]

\begin{lemma}[Linear projection]\label{lemma:2.1}
Let $z\in\mathbb R^p$, $u\in\mathbb R^q$, and consider a finite system
\[
B(x)z+C(x)u\leq d
\]
whose coefficient functions are semialgebraic in $x$. After a finite semialgebraic partition of the $x$-line, the condition
\[
\exists u\in\mathbb R^q
\qquad
B(x)z+C(x)u\leq d
\]
is equivalent to a finite system of linear inequalities in $(z,d)$. All projected coefficients are semialgebraic and depend only on the original coefficient functions.
\end{lemma}

\begin{proof}
Eliminate one coordinate $s$ of $u$ at a time. Write the inequalities involving $s$ as
\[
a_\alpha(x)s\leq b_\alpha(x,z,d).
\]
Partition the line according to the signs of the finitely many semialgebraic functions $a_\alpha$. On each cell, every $a_\alpha$ is either identically zero or has fixed nonzero sign. A row with $a_\alpha>0$ gives an upper bound for $s$, a row with $a_\alpha<0$ gives a lower bound for $s$, and a row with $a_\alpha\equiv0$ gives an inequality independent of $s$.

Existence of $s$ is equivalent to requiring that every lower bound be no larger than every upper bound, together with all inequalities coming from the zero rows. Since the signs of the denominators are fixed on the cell, denominators may be cleared without changing the direction of the corresponding inequalities. The resulting coefficients are again semialgebraic. Iterating this procedure eliminates all coordinates of $u$ and gives the desired finite projected system.
\end{proof}

\begin{corollary}[Pointwise feasibility]\label{cor:2.2}
After a finite semialgebraic partition of $\mathbb R$, the condition
\[
K_f(x)\neq\varnothing
\]
is equivalent, on each cell, to finitely many homogeneous linear inequalities in $f(x)$, with semialgebraic coefficients determined by $A$.
\end{corollary}

\begin{proof}
Apply the preceding lemma to
\[
\exists v\in\mathbb R^M
\qquad
A(x)v\leq f(x).
\]
All operations in Fourier--Motzkin elimination are homogeneous in the right-hand side $f(x)$.

More explicitly, perform Fourier--Motzkin elimination formally on
\[
A(x)v\leq b,
\qquad
b\in\mathbb R^N,
\]
treating $b$ as an independent parameter. Only finitely many semialgebraic coefficient functions of $x$ occur, and the resulting projected conditions are homogeneous linear inequalities in $b$. Substituting $b=f(x)$ gives the assertion.
\end{proof}

\begin{definition}[Regular intervals]\label{def:exceptional set}
Let $\Sigma_A\subset\mathbb R$ be a finite set such that, on every connected component of $\mathbb R\setminus\Sigma_A$, every coefficient arising in the formal Fourier--Motzkin elimination is smooth and either identically zero or of fixed nonzero sign. The components of $\mathbb R\setminus\Sigma_A$ are called the \emph{regular intervals}; the points of $\Sigma_A$ are called the \emph{exceptional points}.
\end{definition}

Such a finite set exists by one-dimensional semialgebraicity and depends only on $A$ \cite{BochnakCosteRoy1998}. Indeed, only finitely many semialgebraic coefficient functions arise in the formal elimination, and each may be partitioned into finitely many intervals and points on which it is smooth and either identically zero or has fixed nonzero sign.

\subsection{The \texorpdfstring{$C^m$}{Cm} scalar insertion theorem}

We use the following consequence of the Riesz decomposition theorem of Nagel and Rudin~\cite{NagelRudin1974}. For $m=0,1,2$, the space $C^m(\mathbb R)$ has the Riesz decomposition property (RD) (see \cite{NagelRudin1974} for the precise definition). Their appendix shows that, for an additive group of real-valued functions, the Riesz decomposition property is equivalent to the weak-lattice property.

\begin{theorem}[Nagel--Rudin interpolation] \label{thm:2.4}
Fix $m\in\{0,1,2\}$. Let
\[
h_1,h_2,H_1,H_2\in C^m(\mathbb R)
\]
satisfy
\[
h_i\leq H_j
\qquad
(i,j=1,2).
\]
Then there exists $\phi\in C^m(\mathbb R)$ such that
\[
h_i\leq\phi\leq H_j
\qquad
(i,j=1,2).
\]
Equivalently,
\[
\max\{h_1,h_2\}
\leq
\phi
\leq
\min\{H_1,H_2\}.
\]
The same statement holds on any open interval $I\subset\mathbb R$.
\end{theorem}

\begin{proof}
For $m=0$, the space of continuous real-valued functions on any topological space is an RD-space. For $m=1,2$ in one dimension, Nagel and Rudin prove that $C^m(\mathbb R)$ is an RD-space. Their appendix proves that, within the class of additive groups of real-valued functions, the RD property is equivalent to the weak-lattice property. The latter is exactly the two-lower/two-upper interpolation statement above.

For an arbitrary open interval $I$, choose a $C^\infty$ diffeomorphism $\Psi:\mathbb R\to I$. Composition with $\Psi$ carries $C^m(I)$ functions to $C^m(\mathbb R)$ functions and preserves pointwise inequalities. Apply the result on $\mathbb R$ and then compose the interpolant with $\Psi^{-1}$.
\end{proof}

\begin{lemma}[Finite $C^m$ scalar insertion]\label{lem:2.5}
Fix $m\in\{0,1,2\}$. Let $I\subset\mathbb R$ be an open interval, and let
\[
L_1,\ldots,L_p,U_1,\ldots,U_q\in C^m(I)
\]
satisfy
\[
L_i(x)\leq U_j(x)
\qquad
(x\in I,\ 1\leq i\leq p,\ 1\leq j\leq q).
\]
Then there exists $h\in C^m(I)$ such that
\[
L_i\leq h\leq U_j
\qquad
(1\leq i\leq p,\ 1\leq j\leq q).
\]
If one of the two families is empty, only the inequalities from the other family are required.
\end{lemma}

\begin{proof}
The empty-family cases are immediate. If both families are empty, take $h=0$. If only the lower family is nonempty, take, for example, $h=L_1$ when $p=1$, and obtain the general finite case by the same induction below with no upper restrictions. The upper-only case is analogous.

Assume now that both families are nonempty. The Nagel--Rudin interpolation theorem gives the interpolation property for two lower and two upper functions. We show that this implies the finite interpolation property by induction.

First fix two lower functions $L_1,L_2$. With one upper function $U_1$, the assertion is immediate: one may take $h=U_1$, since $L_1,L_2\leq U_1$. Suppose that for some $r\geq2$ there is an interpolant $h_{r-1}\in C^m(I)$ satisfying
\[
L_1,L_2\leq h_{r-1}
\]
and
\[
h_{r-1}\leq U_1,\ldots,U_{r-1}.
\]
Apply the two-lower/two-upper interpolation theorem to the lower functions $L_1,L_2$ and the upper functions $h_{r-1},U_r$. Since
\[
L_i\leq h_{r-1}
\qquad\text{and}\qquad
L_i\leq U_r
\]
for $i=1,2$, there exists $h_r\in C^m(I)$ such that
\[
L_1,L_2\leq h_r\leq h_{r-1},U_r.
\]
Because $h_r\leq h_{r-1}\leq U_j$ for $j<r$, the function $h_r$ lies below $U_1,\ldots,U_r$. Thus two lower functions can be interpolated below any finite family of upper functions.

Now induct on the number of lower functions. Suppose $h_0\in C^m(I)$ lies above
\[
L_1,\ldots,L_{p-1}
\]
and below every $U_j$. Apply the two-lower/finite-upper case just proved to the lower functions $h_0,L_p$ and the upper family $U_1,\ldots,U_q$. The resulting function $h$ satisfies
\[
h_0,L_p\leq h\leq U_j
\qquad
(1\leq j\leq q).
\]
Since $h_0\geq L_1,\ldots,L_{p-1}$, it follows that
\[
L_i\leq h\leq U_j
\]
for every $i,j$. This completes the induction.
\end{proof}

\subsection{\texorpdfstring{$C^m$}{Cm} selection on a regular interval}

\begin{proposition}[$C^m$ selection on regular intervals]\label{pro:2.6}
Fix $m\in\{0,1,2\}$. Let $I$ be a regular interval in the sense of Definition \ref{def:exceptional set}. Suppose
\[
K_f(x)\neq\varnothing
\qquad
(x\in I).
\]
Then there exists
\[
F\in C^m(I,\mathbb R^M)
\]
such that
\[
A(x)F(x)\leq f(x)
\qquad
(x\in I).
\]
\end{proposition}

\begin{proof}
We induct on $M$. The assertion is trivial when $M=0$.

Assume the proposition has been proved for $M-1$. Write $v = (v',s)$ with $v' \in \R^{M-1}$ 
and write the $i$-th inequality as
\[
a_i(x)s+c_i(x)\cdot v'\leq f_i(x).
\]
Because $I$ is a regular interval, each coefficient $a_i$ is either identically zero on $I$ or has fixed nonzero sign there.

Rows with $a_i>0$ give upper bounds
\[
s\leq U_i(x,v')
:=
\frac{f_i(x)-c_i(x)\cdot v'}{a_i(x)},
\]
while rows with $a_j<0$ give lower bounds
\[
s\geq L_j(x,v')
:=
\frac{f_j(x)-c_j(x)\cdot v'}{a_j(x)}.
\]
Rows with $a_i\equiv0$ give inequalities involving only $v'$.

Together with the zero rows, the cross inequalities
\[
L_j(x,v')\leq U_i(x,v')
\]
for every lower bound $L_j$ and every upper bound $U_i$ are exactly the Fourier--Motzkin projection of the original system to the first $M-1$ variables. Since the original system is pointwise feasible on $I$, this projected system is also pointwise feasible on $I$.

By the induction hypothesis, the projected system has a selection $F'\in C^m(I,\mathbb R^{M-1})$.
Substituting $v'=F'(x)$ into the scalar bounds gives finitely many functions
\[
L_j(x):=L_j(x,F'(x)),
\qquad
U_i(x):=U_i(x,F'(x)).
\]
These functions belong to $C^m(I)$. Indeed, $f$ is $C^\infty$, the fixed semialgebraic coefficients are smooth on the regular interval, every nonzero denominator $a_i$ is nowhere zero on $I$, and $F'$ is $C^m$.

Because $F'$ satisfies the exact Fourier--Motzkin projection, every lower bound is no larger than every upper bound:
\[
L_j(x)\leq U_i(x)
\]
for all $x\in I$ and for every applicable pair $(j,i)$. Lemma \ref{lem:2.5} therefore gives a function
$F_M\in C^m(I)$
such that
\[
L_j\leq F_M\leq U_i
\]
for all lower and upper bounds. If one of the two families is empty, the corresponding one-sided version of Lemma \ref{lem:2.5} applies.

Set $F=(F',F_M)$.
The zero rows are satisfied because $F'$ satisfies the projected system, and all rows involving the last coordinate are satisfied because $F_M$ lies between every lower and upper bound. Hence
\[
A(x)F(x)\leq f(x)
\qquad
(x\in I).
\]
Thus $F$ is the required feasible $C^m$ selection.
\end{proof}

\section{One-dimensional semialgebraic limits} \label{sec:3}

This section puts together the one-sided semialgebraic facts used at exceptional points and develops the finite-jet test needed uniformly for $C^m$ solvability when
\[
m\in\{0,1,2\}.
\]
The basic Puiseux and order statements are independent of $m$. The dependence on $m$ enters only in the comparison with the order $t^m$.

Fix $x_0\in\mathbb R$. We use $\mathscr G^+_{x_0}$ to denote the ring of right-hand semialgebraic germs at $x_0$. For the standard theory of one-sided semialgebraic germs and their
identification with algebraic Puiseux series, we refer to
\cite[Section~3.3]{BasuPollackRoy2006}.

\begin{lemma}[One-sided Puiseux order]\label{lem:3.1}
Let $\phi\in\mathscr G^+_{x_0}$ be nonzero. Then there exist $c\neq0$ and $\gamma\in\mathbb Q$ such that
\[
\phi(x)
=
c(x-x_0)^\gamma
+
o\bigl((x-x_0)^\gamma\bigr)
\]
as $x\downarrow x_0$. In particular, the sign of $\phi(x)$ is eventually constant as $x\downarrow x_0$.
\end{lemma}

\begin{proof}
This is the standard Puiseux expansion theorem for one-variable semialgebraic functions \cite[Section~3.3]{BasuPollackRoy2006}. After shrinking the right-hand interval, a nonzero semialgebraic function admits a Puiseux expansion in a fractional power of $x-x_0$. The first nonzero term gives the coefficient $c$ and the exponent $\gamma$. Since $x-x_0>0$, the factor $(x-x_0)^\gamma$ is positive for $x>x_0$ sufficiently close to $x_0$. Therefore the eventual sign is the sign of $c$.
\end{proof}

For a nonzero germ $\phi\in\mathscr G^+_{x_0}$, we define its Puiseux order by
\begin{equation}
\operatorname{ord}_{x_0}(\phi)=\gamma,
\label{eq3.1}
\end{equation}
where $\gamma$ is the exponent in Lemma \ref{lem:3.1}. We also set
\[
\operatorname{ord}_{x_0}(0)=+\infty.
\]
In view of Lemma \ref{lem:3.1}, we refer to the sign of the leading coefficient $c$ as the eventual sign of $\phi$ as $x\downarrow x_0$.

\begin{lemma}[Finite order basis]\label{lem:3.2}
Let $V\subset\mathscr G^+_{x_0}$
be a finite-dimensional real vector space of semialgebraic right-hand germs. Then there exist germs
\[
e_1,\ldots,e_s\in V
\]
and rational numbers
\[
\gamma_1<\gamma_2<\cdots<\gamma_s
\]
such that every $\phi\in V$ can be written uniquely as
\[
\phi=\sum_{\mu=1}^s c_\mu e_\mu.
\]
Moreover, if $\phi\neq0$, the eventual sign of $\phi$ is the sign of the first nonzero coefficient $c_\mu$.

Equivalently, there exist linear functionals
\[
\lambda_1,\ldots,\lambda_s:V\to\mathbb R
\]
such that, for every $\phi\in V$, either $\phi=0$, or the eventual sign of $\phi$ is determined by the first nonzero number among
\[
\lambda_1(\phi),\ldots,\lambda_s(\phi).
\]
\end{lemma}

\begin{proof}
We inductively construct a basis using the Puiseux order \eqref{eq3.1}.

Set
\[
W_1=V.
\]
If $W_1=\{0\}$, there is nothing to prove. Otherwise, choose a basis
\[
\{\psi_1,\ldots,\psi_d\}
\]
of $W_1$. Since the finite set
\[
\{\operatorname{ord}_{x_0}(\psi_1),\ldots,\operatorname{ord}_{x_0}(\psi_d)\}
\]
has a minimum, and since
\[
\operatorname{ord}_{x_0}(\phi+\psi)
\geq
\min\{
\operatorname{ord}_{x_0}(\phi),
\operatorname{ord}_{x_0}(\psi)
\},
\]
no linear combination of the $\psi_j$ can have order smaller than this minimum. Hence
\[
\{
\operatorname{ord}_{x_0}(\phi):
\phi\in W_1,\ \phi\neq0
\}
\]
has a minimum, and we set
\[
\gamma_1
:=
\min
\{
\operatorname{ord}_{x_0}(\phi):
\phi\in W_1,\ \phi\neq0
\}.
\]

Choose $e_1\in W_1$ with
\[
\operatorname{ord}_{x_0}(e_1)=\gamma_1.
\]
After multiplying $e_1$ by a nonzero scalar, we may assume
\[
e_1(x)
=
(x-x_0)^{\gamma_1}
+
o\bigl((x-x_0)^{\gamma_1}\bigr).
\]
Define
\[
\alpha_1(\phi)
=
\lim_{x\downarrow x_0}
\frac{\phi(x)}{(x-x_0)^{\gamma_1}},
\qquad
\phi\in W_1.
\]
This limit exists and is finite. It is the leading coefficient if
\[
\operatorname{ord}_{x_0}(\phi)=\gamma_1,
\]
and it is zero if
\[
\operatorname{ord}_{x_0}(\phi)>\gamma_1.
\]
The map
\[
\alpha_1:W_1\to\mathbb R
\]
is linear, and
\[
\alpha_1(e_1)=1.
\]

Set
\[
W_2=\ker\alpha_1.
\]
Then every nonzero element of $W_2$ has order strictly larger than $\gamma_1$. Repeating the same construction inside $W_2$, then inside $W_3$, and so on, we obtain a decreasing sequence
\[
V=W_1\supset W_2\supset\cdots\supset W_s\supset W_{s+1}=\{0\},
\]
germs
\[
e_\mu\in W_\mu,
\qquad
\mu=1,\ldots,s,
\]
and rational numbers
\[
\gamma_1<\gamma_2<\cdots<\gamma_s
\]
such that
\[
W_\mu=\mathbb R e_\mu\oplus W_{\mu+1}.
\]
The process stops after finitely many steps because
\[
\dim W_{\mu+1}<\dim W_\mu
\]
at each nontrivial step. Hence
\[
V
=
\mathbb R e_1\oplus\cdots\oplus\mathbb R e_s.
\]
Therefore every $\phi\in V$ has a unique expansion
\[
\phi=\sum_{\mu=1}^s c_\mu e_\mu.
\]

If $\phi\neq0$, let $\mu_0$ be the first index such that $c_{\mu_0}\neq0$. All terms with index larger than $\mu_0$ have strictly larger order, so
\[
\phi(x)
=
c_{\mu_0}(x-x_0)^{\gamma_{\mu_0}}
+
o\bigl((x-x_0)^{\gamma_{\mu_0}}\bigr).
\]
Since $x-x_0>0$, the eventual sign of $\phi$ is the sign of $c_{\mu_0}$. Taking $\lambda_\mu$ to be the coordinate functional
\[
\lambda_\mu(\phi)=c_\mu
\]
gives the equivalent formulation.
\end{proof}

\subsection{Finite-jet tests for \texorpdfstring{$o(t^m)$}{o(tm)}}

Fix
\[
m\in\{0,1,2\}.
\]
Let $z\in\mathbb R^q$ be a finite-dimensional auxiliary parameter.

\begin{lemma}[Finite-jet positive-part test]\label{lem:3.3}
Fix $y\in\mathbb R$. Suppose
\begin{equation}
\Lambda(t;z,f)
=
\sum_{\alpha=1}^q a_\alpha(t)z_\alpha
+
\sum_{i=1}^N b_i(t)f_i(y+t),
\label{3.2}
\end{equation}
where the $a_\alpha$ and $b_i$ are right-hand semialgebraic germs at $0$. Then there is an integer $s\geq0$, depending only on $m$ and the coefficient germs, such that
\begin{equation}
[\Lambda(t;z,f)]_+
=
o(t^m)
\qquad
(t\downarrow0)
\label{3.3}
\end{equation}
is equivalent to a finite Boolean combination of homogeneous linear sign conditions in
\[
z
\quad\text{and}\quad
J_y^{(s)}f.
\]

Moreover, there is a rational number $\eta>0$, depending only on $m$ and the coefficient germs, such that whenever \eqref{3.3} holds,
\begin{equation}
[\Lambda(t;z,f)]_+
=
O(t^{m+\eta}).
\label{3.4}
\end{equation}
\end{lemma}

\begin{proof}
Choose $K\geq0$ such that every coefficient germ
\begin{equation}
    a_\alpha, \ b_i = O(t^{-K}).
    \label{eq:abOK}
\end{equation}
Choose $s$ so large that
\[
s-K>m+1.
\]
Taylor's theorem gives
\[
f_i(y+t)
=
\sum_{r=0}^s
\frac{f_i^{(r)}(y)}{r!}t^r
+
o(t^s).
\]
Thanks to \eqref{eq:abOK},
the Taylor remainder contributes
\[
b_i(t)o(t^s)
=
o(t^{s-K})
=
o(t^{m+1}).
\]
Therefore
\[
\Lambda(t;z,f)
=
\Phi\bigl(t;z,J_y^{(s)}f\bigr)
+
o(t^{m+1}),
\]
where $\Phi$ belongs to the finite-dimensional real span of the semialgebraic germs
\[
a_\alpha(t)
\quad\text{and}\quad
b_i(t)t^r,
\qquad
0\leq r\leq s.
\]

Apply Lemma \ref{lem:3.2} to this finite-dimensional space. We obtain rational exponents
\[
\gamma_1<\cdots<\gamma_R
\]
and homogeneous linear functionals $L_1,\ldots,L_R$ of $\bigl(z,J_y^{(s)}f\bigr)$
such that, if $\mu$ is the first index for which $L_\mu\neq0$, then
\[
\Phi(t)
=
L_\mu t^{\gamma_\mu}
+
o(t^{\gamma_\mu}).
\]

We compare the first nonzero order with $m$.

Suppose first that
\[
\gamma_\mu\leq m.
\]
If $L_\mu>0$, then
\[
[\Lambda(t;z,f)]_+
\sim
L_\mu t^{\gamma_\mu},
\]
and this is not $o(t^m)$. If $L_\mu<0$, then
\[
\Lambda(t;z,f)<0
\]
for all sufficiently small $t>0$, because the remainder $o(t^{m+1})$ has strictly larger order than $t^{\gamma_\mu}$. Hence
\[
[\Lambda(t;z,f)]_+=0
\]
eventually, and \eqref{3.3} holds.

Thus, whenever the first nonzero order is at most $m$, condition \eqref{3.3} holds exactly when the corresponding first nonzero coefficient is negative.

Suppose instead that every coefficient corresponding to an order at most $m$ vanishes. Then either the first nonzero order is strictly larger than $m$, or $\Phi\equiv0$. In either case
\[
[\Lambda(t;z,f)]_+=o(t^m).
\]
Consequently \eqref{3.3} is equivalent to a finite Boolean combination of the homogeneous linear sign conditions
\[
L_\rho<0,
\qquad
L_\rho=0,
\qquad
L_\rho>0.
\]

It remains to prove the uniform power gain. If the first nonzero order is at most $m$, the criterion forces its coefficient to be negative, and therefore the positive part is eventually zero.

Otherwise all nonzero orders that can govern a positive leading term lie strictly above $m$. Since the list
\[
\gamma_1,\ldots,\gamma_R
\]
is finite, there is a positive gap between $m$ and the first order in this list that is larger than $m$. Choose a rational number $\eta>0$ smaller than this gap and also smaller than $1$. If there is no order larger than $m$, choose any rational $\eta$ with
\[
0<\eta<1.
\]
The term $\Phi$ is then $O(t^{m+\eta})$ whenever its first nonzero order is larger than $m$, while the Taylor remainder
\[
o(t^{m+1})
\]
is also $O(t^{m+\eta})$. Hence
\[
[\Lambda(t;z,f)]_+
=
O(t^{m+\eta}),
\]
which proves \eqref{3.4}.
\end{proof}

\begin{corollary}[Finite families]\label{cor:3.4}
Fix $m\in\{0,1,2\}$ and $y\in\mathbb R$. For a finite family of expressions
\[
\Lambda_\rho(t;z,f)
\]
of the form \eqref{3.2}, there exist a common finite jet order $s$ and a common rational number $\eta>0$ such that the simultaneous conditions
\[
[\Lambda_\rho(t;z,f)]_+
=
o(t^m)
\qquad
(t\downarrow0)
\]
are equivalent to a finite Boolean combination of homogeneous linear sign conditions in
$\bigl(z,J_y^{(s)}f\bigr)$,
and imply
\[
[\Lambda_\rho(t;z,f)]_+
=
O(t^{m+\eta})
\]
for every $\rho$.
\end{corollary}

\begin{proof}
Apply Lemma \ref{lem:3.3} to each member of the finite family. Take $s$ to be the maximum of the finitely many jet orders, adding zero coefficients for the unused higher derivatives, and take $\eta$ to be the minimum of the finitely many positive rational gains. The conjunction of the resulting finite Boolean sign conditions gives the desired common criterion.
\end{proof}

\begin{remark}[Left-hand germs]
All statements in this section have identical left-hand versions after replacing $y+t$ by $y-t$. In particular, for each $m\in\{0,1,2\}$, the condition
\[
[\Lambda_\rho(t;z,f)]_+=o(t^m)
\]
on the left is again finite Boolean-linear in a finite jet of $f$ at $y$, with the same type of positive power gain.
\end{remark}

\section{\texorpdfstring{$C^m$}{Cm} endpoint compatibility}\label{sec:4}

Throughout this section, fix
\[
m\in\{0,1,2\}.
\]
Fix an exceptional point $y\in\Sigma_A$. We characterize those polynomial $m$-jets that can be realized by feasible one-sided $C^m$ selections.

\subsection{Distance to the feasible polyhedron}

Suppose a regular interval lies immediately to the right of $y$. For $P\in\mathcal P^m(\mathbb R,\mathbb R^M)$, define (cf. \eqref{1.5})
\begin{equation}
d^+_{y,P}(t)
=
\operatorname{dist}_\infty
\bigl(
P(y+t),K_f(y+t)
\bigr),
\qquad
t>0.
\label{eq:distance-plus}
\end{equation}
Define $d^-_{y,P}$ similarly using $y-t$.

\begin{lemma}[Dual formula for the one-point distance]\label{lem:4.1}
Fix an exceptional point $y\in\Sigma_A$. Assume
\[
K_f(y+t)
:=
\{v\in\mathbb R^M:A(y+t)v\leq f(y+t)\}
\neq\varnothing
\]
for all sufficiently small $t>0$. Then, after shrinking the interval, there exist finitely many semialgebraic functions
\[
\lambda_\alpha(t)\in[0,\infty)^N,
\qquad
\alpha=1,\ldots,L,
\]
depending only on the germ of $A$, such that
\[
d^+_{y,P}(t)
=
\max
\{
0,
\Lambda_1(t;P,f),
\ldots,
\Lambda_L(t;P,f)
\},
\]
where
\[
\Lambda_\alpha(t;P,f)
:=
\lambda_\alpha(t)^T
\bigl(
A(y+t)P(y+t)-f(y+t)
\bigr).
\]
\end{lemma}

\begin{proof}
Fix $t>0$ sufficiently small, and abbreviate
\[
A=A(y+t),
\qquad
p=P(y+t),
\qquad
f=f(y+t),
\]
and
\[
K:=\{v\in\mathbb R^M:Av\leq f\}.
\]
By assumption, $K\neq\varnothing$. Set
\[
d:=\operatorname{dist}_\infty(p,K).
\]

For $s\geq0$, the inequality $d\leq s$ is equivalent to the existence of $v\in K$ such that
\[
\|v-p\|_\infty\leq s.
\]
Writing
\[
v=p+r,
\]
this is equivalent to
\begin{equation}
Ar\leq f-Ap,
\qquad
r\leq s\mathbf 1,
\qquad
-r\leq s\mathbf 1.
\label{eq:farkas-primal}
\end{equation}
Here $\mathbf 1$ denotes the vector all of whose components are equal to $1$.

By Farkas' lemma (see e.g. \cite{Rockafellar1970}), the system \eqref{eq:farkas-primal} is feasible if and only if
\begin{equation}
\lambda^T(f-Ap)
+
s\mathbf 1^T\mu
+
s\mathbf 1^T\nu
\geq0
\label{eq:farkas-dual}
\end{equation}
for every
\[
\lambda\in[0,\infty)^N,
\qquad
\mu,\nu\in[0,\infty)^M
\]
satisfying
\begin{equation}
A^T\lambda+\mu-\nu=0.
\label{eq:farkas-constraint}
\end{equation}

Fix $\lambda\geq0$. Then
\[
\mu-\nu=-A^T\lambda.
\]
For a scalar $a$,
\[
\min_{\substack{\mu,\nu\geq0\\ \mu-\nu=-a}}
(\mu+\nu)
=
|a|.
\]
Applying this coordinatewise yields
\[
\min_{\substack{\mu,\nu\geq0\\ \mu-\nu=-A^T\lambda}}
\mathbf 1^T(\mu+\nu)
=
\|A^T\lambda\|_1.
\]
Therefore \eqref{eq:farkas-dual} is equivalent to
\begin{equation}
\lambda^T(f-Ap)
+
s\|A^T\lambda\|_1
\geq0
\qquad
\text{for every }\lambda\geq0.
\label{eq:farkas-reduced}
\end{equation}

If
\[
A^T\lambda=0,
\]
then \eqref{eq:farkas-reduced} is automatic. Indeed, choose $v_0\in K$. Since
\[
Av_0\leq f
\]
and $\lambda\geq0$,
\[
0
=
\lambda^TAv_0
\leq
\lambda^Tf.
\]
Moreover,
\[
\lambda^TAp
=
(A^T\lambda)^Tp
=
0.
\]
Hence
\[
\lambda^T(f-Ap)\geq0.
\]

Thus only multipliers satisfying $A^T\lambda\neq0$ can impose a restriction on $s$. For such $\lambda$, \eqref{eq:farkas-reduced} is equivalent to
\[
s
\geq
\frac{
[\lambda^T(Ap-f)]_+
}{
\|A^T\lambda\|_1
}.
\]
Consequently,
\begin{equation}
d
=
\sup_{\substack{\lambda\geq0\\ A^T\lambda\neq0}}
\frac{
[\lambda^T(Ap-f)]_+
}{
\|A^T\lambda\|_1
}.
\label{eq:dual-distance}
\end{equation}
Indeed, every admissible radius $s$ is at least the right-hand side. Conversely, if $s$ is strictly larger than that supremum, then \eqref{eq:farkas-reduced} holds for every $\lambda\geq0$. Farkas' lemma therefore gives a correction $r$ satisfying \eqref{eq:farkas-primal}, and hence $d\leq s$. Letting $s$ decrease to the supremum gives equality.

The quotient in \eqref{eq:dual-distance} is homogeneous of degree zero in $\lambda$. Hence
\begin{equation}
d
=
\sup_{\substack{\lambda\geq0\\ \|A^T\lambda\|_1=1}}
[\lambda^T(Ap-f)]_+.
\label{eq:normalized-dual}
\end{equation}

We now split the normalized multiplier set according to the sign pattern of $A^T\lambda$. For
\[
\sigma=(\sigma_1,\ldots,\sigma_M)\in\{-1,1\}^M,
\]
define
\[
Q_\sigma(t)
:=
\left\{
\lambda\in\mathbb R^N:
\lambda\geq0,\ 
\sigma_j(A^T\lambda)_j\geq0\ \text{for }1\leq j\leq M,\ 
\sum_{j=1}^M\sigma_j(A^T\lambda)_j=1
\right\}.
\]
On $Q_\sigma(t)$,
\[
\sum_{j=1}^M
\sigma_j(A^T\lambda)_j
=
\|A^T\lambda\|_1.
\]
As a consequence,
\[
\{\lambda\geq0:\|A^T\lambda\|_1=1\}
=
\bigcup_{\sigma\in\{-1,1\}^M}
Q_\sigma(t).
\]

Each $Q_\sigma(t)$ is a polyhedron in $\lambda$-space. We maximize the linear functional
\[
\lambda
\longmapsto
\lambda^T(Ap-f)
\]
over $Q_\sigma(t)$.

Let $\rho$ be a recession direction of $Q_\sigma(t)$ (see e.g. \cite{Rockafellar1970}). Then $\rho\geq0.$.
Since the normalization remains equal to $1$ along $\lambda+s\rho$, we have
\[
\sum_{j=1}^M
\sigma_j(A^T\rho)_j
=
0.
\]
The sign inequalities also give
\[
\sigma_j(A^T\rho)_j\geq0
\qquad
(1\leq j\leq M),
\]
and therefore
\[
A^T\rho=0.
\]
As above, pointwise feasibility of $K$ implies
\[
\rho^Tf\geq0,
\]
while
\[
\rho^TAp=0.
\]
Hence
\[
\rho^T(Ap-f)\leq0.
\]
Thus no recession direction can increase the objective. It follows that the finite maximum may be attained at an extreme point of $Q_\sigma(t)$.

An extreme point is determined by the normalization equality together with enough active inequalities from
\[
\lambda_i\geq0
\text{\quad and \quad}
\sigma_j(A^T\lambda)_j\geq0
\]
to obtain $N$ linearly independent equalities. There are only finitely many possible choices of active constraints. For each such choice $\beta$, the corresponding candidate is determined by a square system
\[
M_\beta(t)\lambda=c_\beta,
\]
whose coefficients are semialgebraic functions of $t$.

After shrinking the interval, every determinant $\det M_\beta(t)$
is either identically zero or nowhere zero. In the latter case, Cramer's rule gives
\[
\lambda_\beta(t)
=
M_\beta(t)^{-1}c_\beta,
\]
and every component of $\lambda_\beta(t)$ is semialgebraic. Since there are only finitely many sign patterns and finitely many choices of active constraints, after shrinking the interval once more we obtain finitely many semialgebraic multiplier functions
\[
\lambda_1(t),\ldots,\lambda_L(t)\in[0,\infty)^N
\]
such that
\[
d
=
\max
\left\{
0,
\lambda_1(t)^T(Ap-f),
\ldots,
\lambda_L(t)^T(Ap-f)
\right\}.
\]

Restoring the notation
\[
A=A(y+t),
\qquad
p=P(y+t),
\qquad
f=f(y+t),
\]
and defining
\[
\Lambda_\alpha(t;P,f)
=
\lambda_\alpha(t)^T
\bigl(
A(y+t)P(y+t)-f(y+t)
\bigr),
\]
gives
\[
d^+_{y,P}(t)
=
\max
\{
0,
\Lambda_1(t;P,f),
\ldots,
\Lambda_L(t;P,f)
\},
\]
as required.
\end{proof}

\begin{corollary}[Semialgebraic error bound]\label{cor:4.2}
Let $A(t)$ be a matrix of one-sided semialgebraic germs. Then there exist constants $C>0$ and $K\geq0$, depending only on $A$, such that whenever
\[
K_b(t)
:=
\{u\in\mathbb R^M:A(t)u\leq b\}
\]
is nonempty, one has
\[
\operatorname{dist}_\infty
\bigl(
v,K_b(t)
\bigr)
\leq
Ct^{-K}
\|
(A(t)v-b)_+
\|_\infty
\]
for all sufficiently small $t>0$ and all $v\in\mathbb R^M$.
\end{corollary}

\begin{proof}
Apply Lemma \ref{lem:4.1} to the matrix $A(t)$. After shrinking the interval, there are finitely many semialgebraic multiplier functions
\[
\lambda_\alpha(t)\in[0,\infty)^N,
\qquad
\alpha=1,\ldots,L,
\]
depending only on $A$, such that
\[
\operatorname{dist}_\infty
\bigl(
v,K_b(t)
\bigr)
=
\max_\alpha
\left[
\lambda_\alpha(t)^T
\bigl(
A(t)v-b
\bigr)
\right]_+.
\]
Since $\lambda_\alpha(t)\geq0$,
\[
\left[
\lambda_\alpha(t)^T
\bigl(
A(t)v-b
\bigr)
\right]_+
\leq
\lambda_\alpha(t)^T
\bigl(
A(t)v-b
\bigr)_+,
\]
and therefore
\[
\left[
\lambda_\alpha(t)^T
\bigl(
A(t)v-b
\bigr)
\right]_+
\leq
\|\lambda_\alpha(t)\|_1
\|
(A(t)v-b)_+
\|_\infty.
\]
Hence
\[
\operatorname{dist}_\infty
\bigl(
v,K_b(t)
\bigr)
\leq
\left(
\max_\alpha
\|\lambda_\alpha(t)\|_1
\right)
\|
(A(t)v-b)_+
\|_\infty.
\]

Each component of each $\lambda_\alpha(t)$ is a one-variable semialgebraic germ, and therefore has at most polynomial growth as $t\downarrow0$. Since there are only finitely many $\lambda_\alpha$, there exist $C>0$ and $K\geq0$, depending only on $A$, such that
\[
\max_\alpha
\|\lambda_\alpha(t)\|_1
\leq
Ct^{-K}.
\]
The desired estimate follows.
\end{proof}

\begin{corollary}[Finite description and power gain]\label{cor:4.3}
Fix $m\in\{0,1,2\}$. For each exceptional point $y$ and each applicable side, there is an integer $s_y$ such that
\[
d^\pm_{y,P}(t)=o(t^m)
\]
is equivalent to a finite Boolean combination of homogeneous linear sign conditions in
\[
P
\quad\text{and}\quad
J_y^{(s_y)}f.
\]
Moreover, for some rational $\eta>0$ depending only on $m$ and the fixed coefficient germs,
\[
d^\pm_{y,P}(t)=o(t^m)
\quad\Longrightarrow\quad
d^\pm_{y,P}(t)=O(t^{m+\eta}).
\]
\end{corollary}

\begin{proof}
Use Lemma \ref{lem:4.1} and Corollary \ref{cor:3.4}. The coefficients of the polynomial $P$ form a finite-dimensional auxiliary parameter, and the expressions supplied by Lemma \ref{lem:4.1} are of the form treated in Section \ref{sec:3}.
\end{proof}

\subsection{Weighted \texorpdfstring{$C^m$}{Cm} interpolation}

For $\lambda\in\mathbb R$, let $X^m_\lambda$ be the class of functions $u\in C^m((0,\varepsilon))$
satisfying
\begin{equation}
|u^{(r)}(t)|
\leq
Ct^{\lambda-r},
\qquad
0\leq r\leq m,
\label{eq:weighted-space}
\end{equation}
for some constant $C$ and all sufficiently small $t>0$. If
\[
\lambda>m,
\]
then every $u\in X^m_\lambda$ extends to a $C^m$ function at $0$ with zero $m$-jet. Indeed, \eqref{eq:weighted-space} gives
\[
u^{(r)}(t)=O(t^{\lambda-r})\longrightarrow0
\qquad
(0\leq r\leq m),
\]
and the successive derivatives therefore extend continuously by zero.

If $a$ is a nonzero one-sided semialgebraic germ, then after shrinking the interval,
\[
a^{(r)}(t)
=
O\bigl(
t^{\operatorname{ord}(a)-r}
\bigr),
\qquad
0\leq r\leq m,
\]
and
\[
(a^{-1})^{(r)}(t)
=
O\bigl(
t^{-\operatorname{ord}(a)-r}
\bigr),
\qquad
0\leq r\leq m.
\]
Consequently, by Leibniz' rule,
\begin{equation}
aX^m_\lambda
\subset
X^m_{\lambda+\operatorname{ord}(a)},
\qquad
a^{-1}X^m_\lambda
\subset
X^m_{\lambda-\operatorname{ord}(a)}.
\label{eq:weighted-multiplication}
\end{equation}

\begin{lemma}[Weighted $C^m$ scalar insertion]\label{lem:4.4}
Fix $m\in\{0,1,2\}$. Let
\[
L_1,\ldots,L_p,U_1,\ldots,U_q
\in
X^m_\lambda
\]
and suppose
\[
L_i\leq U_j
\text{\quad for all $i,j$ and all sufficiently small $t>0$.}
\]
Then there exists $h\in X^m_\lambda$
such that
\[
L_i\leq h\leq U_j \text{\quad for every $i,j$.}
\]

\end{lemma}

\begin{proof}
Set $t=e^{-x}$
and define
\[
\widetilde u(x)
=
e^{\lambda x}u(e^{-x}).
\]
For $m=0$, the condition $u\in X^0_\lambda$ is exactly the boundedness of $\widetilde u$ on $[0,R]$ for some large $R$.

For $m\geq1$, direct differentiation gives
\[
\widetilde u'(x)
=
e^{\lambda x}
\bigl(
\lambda u-tu'
\bigr).
\]
For $m=2$, differentiating once more gives
\[
\widetilde u''(x)
=
e^{\lambda x}
\bigl(
\lambda^2u
-
(2\lambda-1)tu'
+
t^2u''
\bigr).
\]
These identities are triangular in
\[
e^{\lambda x}u,
\qquad
e^{\lambda x}tu',
\qquad
e^{\lambda x}t^2u'',
\]
and therefore show that, for $m=0,1,2$, $u\in X^m_\lambda$
if and only if $\widetilde u,
\widetilde u',
\ldots,
\widetilde u^{(m)}$
are uniformly bounded on $[0,R]$ for some large $R$.

For $m=0$, bounded continuous functions form a lattice, so the bounded weak-lattice interpolation property is immediate. For $m=1,2$, the one-dimensional Riesz decomposition results of Nagel and Rudin \cite{NagelRudin1974}, together with their explicit Riesz construction, give the corresponding bounded weak-lattice interpolation property: if the data and their derivatives through order $m$ are uniformly bounded, the interpolant may be chosen with derivatives through order $m$ uniformly bounded.

Apply this bounded interpolation property to the transformed lower and upper functions $\widetilde L_i$ and $\widetilde U_j$,
using the finite-family induction from Lemma \ref{lem:2.5}. We obtain a transformed interpolant $\widetilde h$ whose derivatives through order $m$ are uniformly bounded and which satisfies
\[
\widetilde L_i
\leq
\widetilde h
\leq
\widetilde U_j.
\]
Multiplication by
\[
e^{-\lambda x}=t^\lambda
\]
gives the required function $h\in X^m_\lambda$.
\end{proof}

\begin{lemma}[Weighted boxed selection]\label{lem:4.5}
Fix $m\in\{0,1,2\}$. Let
\[
B(t)
=
\bigl(
B_{ij}(t)
\bigr)_{1\leq i\leq L,\ 1\leq j\leq M}
\]
be a matrix of one-sided semialgebraic germs on $t>0$. Then there exists a constant
\[
D=D(B,M,m)\geq0
\]
with the following property.

Let $\lambda\in\mathbb R$, let $\Lambda\geq\lambda$, and suppose
\[
g_i\in X^m_\lambda,
\qquad
i=1,\ldots,L.
\]
Assume that, for every sufficiently small $t>0$, the polyhedron
\[
K(t)
=
\left\{
v\in\mathbb R^M:
B(t)v\leq g(t),
\quad
|v_j|\leq C_0t^\Lambda
\ \text{for }j=1,\ldots,M
\right\}
\]
is nonempty. Then, after decreasing the interval if necessary, there exists
\[
v=(v_1,\ldots,v_M)
\in
C^m((0,\varepsilon),\mathbb R^M)
\]
such that
\[
v(t)\in K(t)
\]
for all sufficiently small $t>0$, and
\[
v_j\in X^m_{\lambda-D},
\qquad
j=1,\ldots,M.
\]
\end{lemma}

\begin{proof}
We argue by induction on $M$. The assertion is trivial for $M=0$. Assume it has been proved for $M-1$, and write $v=(v',s)$ with $v' \in \R^{M-1}$.

Let $a_i=B_{iM}$. We write the $i$-th inequality as
\begin{equation}
a_i(t)s+b_i(t)\cdot v'\leq g_i(t),
\label{eq:boxed-row}
\end{equation}
After shrinking the interval, every nonzero semialgebraic germ $a_i$ has a fixed sign. Thus the rows split into three classes:
\[
a_i>0,
\qquad
a_i<0,
\qquad
a_i\equiv0.
\]

For $a_i>0$, \eqref{eq:boxed-row} gives the upper bound
\[
s
\leq
U_i(t,v')
:=
\frac{
g_i(t)-b_i(t)\cdot v'
}{
a_i(t)
},
\]
while for $a_j<0$ it gives the lower bound
\[
s
\geq
L_j(t,v')
:=
\frac{
g_j(t)-b_j(t)\cdot v'
}{
a_j(t)
}.
\]
We also have the box bounds
\[
-C_0t^\Lambda
\leq
s
\leq
C_0t^\Lambda.
\]

Hence the existence of $s$ is equivalent to the following system in $v'$:

\begin{enumerate}
\item the rows with $a_i\equiv0$,
\[
b_i(t)\cdot v'\leq g_i(t);
\]

\item every lower bound is no larger than every upper bound,
\[
L_j(t,v')\leq U_i(t,v');
\]

\item every lower bound is no larger than the box upper bound,
\[
L_j(t,v')\leq C_0t^\Lambda;
\]

\item the box lower bound is no larger than every upper bound,
\[
-C_0t^\Lambda\leq U_i(t,v');
\]

\item the remaining coordinate box constraints,
\[
|v'_k|
\leq
C_0t^\Lambda,
\qquad
k=1,\ldots,M-1.
\]
\end{enumerate}

This is exactly the Fourier--Motzkin projection of $K(t)$ onto the first $M-1$ coordinates.

We next keep track of weighted orders. For every nonzero semialgebraic germ $c(t)$, \eqref{eq:weighted-multiplication} gives
\[
cX^m_\mu
\subset
X^m_{\mu+\operatorname{ord}(c)},
\qquad
c^{-1}X^m_\mu
\subset
X^m_{\mu-\operatorname{ord}(c)}.
\]
Thus multiplication or division by a fixed semialgebraic germ changes the weighted order by only a fixed amount.

The Fourier--Motzkin projection above uses only finitely many additions, multiplications by fixed semialgebraic germs, and divisions by the nonzero germs $a_i$. Therefore there exists a constant
\[
D_1=D_1(B,m)\geq0
\]
such that every right-hand side occurring in the projected system belongs to $X^m_{\lambda-D_1}$.
The box term $C_0t^\Lambda$
also belongs to this class, since $\Lambda\geq\lambda\geq\lambda-D_1$.

The projected system is pointwise nonempty because $K(t)\neq\varnothing$. By the induction hypothesis, there is a $C^m$ selection $v'(t)$
satisfying the projected system and
\[
v'_k
\in
X^m_{\lambda-D_1-D_2},
\qquad
k=1,\ldots,M-1,
\]
where $D_2$ depends only on the projected coefficient matrix, $M$, and $m$, and hence ultimately only on $B$, $M$, and $m$.

Now substitute $v'(t)$ into the original inequalities. We obtain finitely many scalar lower bounds
\[
L_j(t)
=
\frac{
g_j(t)-b_j(t)\cdot v'(t)
}{
a_j(t)
}
\]
and upper bounds
\[
U_i(t)
=
\frac{
g_i(t)-b_i(t)\cdot v'(t)
}{
a_i(t)
},
\]
together with
\[
-C_0t^\Lambda
\leq
s
\leq
C_0t^\Lambda.
\]

Again, only finitely many additions, multiplications by fixed semialgebraic germs, and divisions by the fixed nonzero $a_i$ occur. Hence there is another constant
\[
D_3=D_3(B,M,m)\geq0
\]
such that all of these lower and upper bounds belong to $X^m_{\lambda-D_1-D_2-D_3}$.

Because $v'(t)$ satisfies the exact Fourier--Motzkin projection, every lower bound is no larger than every upper bound. Therefore, Lemma \ref{lem:4.4} gives
$s
\in
X^m_{\lambda-D_1-D_2-D_3}$
satisfying all the scalar bounds. Thus
\[
v(t)
=
(v'(t),s(t)) \in K(t)
\text{\quad for all sufficiently small $t>0$.}
\]

Finally, since the Fourier--Motzkin procedure has only finitely many steps, the total loss in weighted order is bounded by a constant depending only on the fixed semialgebraic coefficient matrix $B$, the dimension $M$, and $m$. Taking
\[
D
\geq
D_1+D_2+D_3
\]
large enough to dominate all losses arising in the induction gives
\[
v_j\in X^m_{\lambda-D},
\qquad
j=1,\ldots,M,
\]
as required.
\end{proof}

\subsection{The one-sided \texorpdfstring{$C^m$}{Cm} endpoint theorem}

We now state the main result regarding exceptional points. It completely characterizes which polynomial $m$-jets arise from feasible one-sided $C^m$ selections for
\[
m\in\{0,1,2\}.
\]

\begin{theorem}[$C^m$ flat-correction criterion]\label{thm:4.6}
Fix $m\in\{0,1,2\}$. Let $M,N\geq1$, and let
\[
A
=
(A_{ij})_{1\leq i\leq N,\ 1\leq j\leq M}
:
\mathbb R
\to
\mathbb R^{N\times M}
\]
be a matrix of univariate semialgebraic functions. Let
\[
f=(f_1,\ldots,f_N)
\in
C^\infty(\mathbb R,\mathbb R^N),
\]
and define, for each $x\in\mathbb R$,
\[
K_f(x)
:=
\{v\in\mathbb R^M:A(x)v\leq f(x)\},
\]
where the inequality is understood coordinatewise.

Let $\Sigma_A$ be the finite exceptional set associated with $A$ in the sense of Definition~\ref{def:exceptional set}, and fix $y\in\Sigma_A$.
Assume that a regular interval lies immediately to the right of $y$, and that for some $\varepsilon>0$,
\[
K_f(y+t)\neq\varnothing
\qquad
(0<t<\varepsilon).
\]
Let $P\in\mathcal P^m(\mathbb R,\mathbb R^M)$ where $\mathcal P^m(\mathbb R,\mathbb R^M)$ denotes the space of $\mathbb R^M$-valued polynomials of degree at most $m$. Then the following are equivalent.

\begin{enumerate}
\item[(i)] There exist $\varepsilon'>0$ and
\[
F
\in
C^m([y,y+\varepsilon'),\mathbb R^M)
\]
such that
\[
J_y^mF=P
\]
and
\[
A(x)F(x)\leq f(x)
\qquad
(y\leq x<y+\varepsilon').
\]

\item[(ii)] One has
\[
P(y)\in K_f(y)
\]
and
\[
d^+_{y,P}(t)
:=
\operatorname{dist}_\infty
\bigl(
P(y+t),K_f(y+t)
\bigr)
=
o(t^m)
\qquad
(t\downarrow0),
\]
where
\[
\operatorname{dist}_\infty(u,K)
:=
\inf_{v\in K}\|u-v\|_\infty,
\qquad
\|w\|_\infty
:=
\max_{1\leq j\leq M}|w_j|.
\]
For $m=0$, the condition $o(t^m)$ means $o(1)$.
\end{enumerate}

The analogous statement holds from the left.
\end{theorem}

\begin{proof}
We prove the right-hand statement.

Assume first that \emph{(i)} holds. Since
\[
J_y^mF=P,
\]
Taylor's theorem gives
\[
F(y+t)-P(y+t)
=
o(t^m).
\]
Because
\[
F(y+t)\in K_f(y+t),
\]
we obtain
\[
d^+_{y,P}(t)
\leq
\|F(y+t)-P(y+t)\|_\infty
=
o(t^m).
\]
Also
\[
P(y)=F(y)\in K_f(y).
\]
Thus \emph{(ii)} holds.

Conversely, assume \emph{(ii)}. By Corollary \ref{cor:4.3}, there are $\eta>0$ and $C>0$ such that
\begin{equation}
d^+_{y,P}(t)
\leq
Ct^{m+\eta}
\label{eq:distance-power-gain}
\end{equation}
for all sufficiently small $t>0$.

For notational simplicity, we write
\[
A(t)=A(y+t),
\qquad
f(t)=f(y+t).
\]
Choose a large integer $L$. Let
\[
T_Lf(t)
=
\sum_{r=0}^L
\frac{f^{(r)}(y)}{r!}t^r.
\]
Taylor's theorem gives a constant $C_L$ such that
\[
|f_i(t)-(T_Lf)_i(t)|
\leq
C_Lt^{L+1}
\]
for every $i$ and all sufficiently small $t>0$.

Define the relaxed Taylor data
\begin{equation}
f^{[L]}(t)
=
T_Lf(t)
+
C_Lt^{L+1}\mathbf 1.
\label{eq:relaxed-taylor}
\end{equation}
Then
\[
f(t)\leq f^{[L]}(t)
\]
coordinatewise. Hence
\[
K_f(y+t)
\subset
K_{f^{[L]}}(y+t).
\]
Therefore \eqref{eq:distance-power-gain} gives
\[
\operatorname{dist}_\infty
\bigl(
P(y+t),K_{f^{[L]}}(y+t)
\bigr)
\leq
Ct^{m+\eta}.
\]

Consider the semialgebraic set
\[
\mathcal R_0
=
\left\{
(t,R)\in\mathbb R\times\mathbb R^M:
0<t<\varepsilon,\ 
A(t)\bigl(P(y+t)+R\bigr)\leq f^{[L]}(t),\
\|R\|_\infty\leq2Ct^{m+\eta}
\right\}.
\]
After shrinking $\varepsilon$ if necessary, every fiber of $\mathcal R_0$ is nonempty. By semialgebraic choice, there is a semialgebraic selection $R_0(t)$
with
\[
(t,R_0(t))\in\mathcal R_0.
\]
After shrinking the interval again, $R_0$ is smooth for $t>0$, and
\[
\|R_0(t)\|_\infty
=
O(t^{m+\eta}).
\]
By the Puiseux theorem, every component of $R_0$ has Puiseux order at least $m+\eta$. Hence, for $0\leq r\leq m$,
\[
R_0^{(r)}(t)
=
O(t^{m+\eta-r})
\longrightarrow0.
\]
Thus $R_0$ extends to a $C^m$ function at $0$ with zero $m$-jet.

Next, define 
\begin{equation}
w(t)
=
f^{[L]}(t)
-
A(t)\bigl(P(y+t)+R_0(t)\bigr)
\geq0
\label{eq:slack}
\end{equation}
and put
\[
q(t)
=
f(t)-f^{[L]}(t).
\]
Taylor's theorem, applied also to the derivatives through order $m$, gives
\begin{equation}
q_i\in X^m_{L+1},
\qquad
-2C_Lt^{L+1}
\leq
q_i(t)
\leq
0.
\label{eq:q-estimate}
\end{equation}
It remains to find $E(t)$ satisfying
\begin{equation}
A(t)E(t)
\leq
w(t)+q(t).
\label{eq:E-target}
\end{equation}

Set
\[
p_0(t)
=
P(y+t)+R_0(t).
\]
From \eqref{eq:slack},
\[
A(t)p_0(t)
=
f^{[L]}(t)-w(t).
\]
Therefore
\[
A(t)p_0(t)-f(t)
=
-w(t)-q(t).
\]
Since $w(t)\geq0$ and $q_i(t)\geq-2C_Lt^{L+1}$,
\[
\|
(A(t)p_0(t)-f(t))_+
\|_\infty
\leq
2C_Lt^{L+1}.
\]

By Corollary \ref{cor:4.2}, there are constants $K\geq0$ and $C'>0$, depending only on $A$, such that pointwise feasibility of the original system gives, for every sufficiently small $t>0$, a point
\[
v^*(t)\in K_f(y+t)
\]
with
\[
\|v^*(t)-p_0(t)\|_\infty
\leq
C't^{L+1-K}.
\]
Define the pointwise correction
\[
E^*(t)
=
v^*(t)-p_0(t).
\]
No regularity of $E^*$ is required. Since $v^*(t)\in K_f(y+t)$,
\[
A(t)E^*(t)
\leq
f(t)-A(t)p_0(t)
=
w(t)+q(t).
\]
Moreover,
\begin{equation}
\|E^*(t)\|_\infty
\leq
C't^\Lambda,
\qquad
\Lambda=L+1-K.
\label{eq:E-star-box}
\end{equation}

For each row $i$, let
\[
\beta_i := \begin{cases}
    \operatorname{ord}(w_i) &\text{when $w_i\not\equiv0$}
    \\
    +\infty &\text{when $w_i\equiv0$}
\end{cases} \ .
\]
For each nonzero row of $A$, let
\[
\alpha_i
=
\min_{j:A_{ij}\not\equiv0}
\operatorname{ord}(A_{ij}).
\]
Rows that are identically zero are treated separately.

Choose a constant $C_1\geq0$, depending only on $A$ and $K$, so large that, with $\theta=L-C_1$, we have
\[
\Lambda\geq\theta
\text{\quad and \quad }
\Lambda+\alpha_i>\theta
\]
for every nonzero row. This choice is possible independently of $L$, because these conditions are equivalent to
\[
C_1\geq K-1
\text{\quad and \quad }
C_1>K-1-\alpha_i
\]
for the finitely many nonzero rows.

Suppose first that
\[
\beta_i<\theta.
\]
Since $w_i\geq0$ and $w_i\not\equiv0$, the Puiseux theorem gives
\[
w_i(t)
\sim
c_it^{\beta_i},
\qquad
c_i>0.
\]
On the box
\[
\|E\|_\infty
\leq
C't^\Lambda,
\]
the choice of $\theta$ gives
\[
A_i(t)E
=
o(t^\theta).
\]
Also, by \eqref{eq:q-estimate},
\[
q_i(t)
=
O(t^{L+1})
=
o(t^\theta).
\]
Since $\beta_i<\theta$,
we have
\[
t^\theta=o(t^{\beta_i}),
\]
and therefore, uniformly for all $E$ in the box,
\[
|A_i(t)E|+|q_i(t)|
=
o(w_i(t)).
\]
Hence the $i$-th inequality in \eqref{eq:E-target} is automatic for every $E$ in the box and all sufficiently small $t>0$.

Now consider the remaining rows, for which
\[
\beta_i\geq\theta.
\]
If $w_i\not\equiv0$, the semialgebraic Puiseux estimates give
\[
w_i\in X^m_{\beta_i}
\subset
X^m_\theta.
\]
If $w_i\equiv0$, then trivially
\[
w_i\in X^m_\theta.
\]
Also, since $C_1\geq0$,
\[
L+1>\theta,
\]
so \eqref{eq:q-estimate} gives
\[
q_i\in X^m_{L+1}
\subset
X^m_\theta.
\]
Therefore
\[
w_i+q_i
\in
X^m_\theta
\]
for every remaining row.

Consider the subsystem consisting of these remaining rows together with the box constraints
\[
|E_j|
\leq
C't^\Lambda,
\qquad
j=1,\ldots,M.
\]
It is pointwise nonempty because $E^*(t)$ satisfies it. Since $\Lambda\geq\theta$,
Lemma \ref{lem:4.5} applies with $\lambda=\theta$.
Thus there is a constant
\[
D=D(A,M,m),
\]
independent of $L$, and a $C^m$ selection
\[
E(t)
\]
satisfying the subsystem and the box, with
\[
E_j\in X^m_{\theta-D},
\qquad
j=1,\ldots,M.
\]

Choose $L$ so large that
\[
\theta-D>m.
\]
Since
\[
\theta=L-C_1
\]
and both $C_1$ and $D$ are independent of $L$, this is possible. Then every component of $E$ extends to $t=0$ as a $C^m$ function with zero $m$-jet. The rows discarded above are automatic throughout the box, so $E$ satisfies the full system \eqref{eq:E-target}.

Finally set
\[
R(t)
=
R_0(t)+E(t)
\]
and
\[
F(y+t)
=
P(y+t)+R(t),
\qquad
F(y)=P(y).
\]
Both $R_0$ and $E$ are $C^m$ at $0$ with zero $m$-jet. Hence
\[
F\in C^m([y,y+\varepsilon'),\mathbb R^M)
\]
for a sufficiently small $\varepsilon'>0$, and
\[
J_y^mF=P.
\]
Moreover,
\[
\begin{aligned}
A(t)F(y+t)
&=
A(t)
\bigl(
P(y+t)+R_0(t)+E(t)
\bigr)
\\
&\leq
f^{[L]}(t)-w(t)+w(t)+q(t)
\\
&=
f(t).
\end{aligned}
\]
At $t=0$, feasibility follows from
\[
P(y)\in K_f(y).
\]
This proves the right-hand implication. The left-hand case is identical after replacing $y+t$ by $y-t$.
\end{proof}

\subsection{Finite exceptional-point conditions}

\begin{proposition}[Finite one-sided endpoint description]\label{pro:4.7}
Fix $m\in\{0,1,2\}$. For every $y\in\Sigma_A$ and each side
\[
\varepsilon\in\{-,+\},
\]
there is an integer $s_y$ and a finite Boolean formula
\[
\mathcal B^\varepsilon_y
\bigl(
P,J_y^{(s_y)}f
\bigr)
\]
built from homogeneous linear sign conditions such that, assuming pointwise feasibility on the corresponding punctured regular interval, $P\in\mathcal P^m(\mathbb R,\mathbb R^M)$
is realized by a feasible one-sided $C^m$ selection if and only if
$\mathcal B^\varepsilon_y
\bigl(
P,J_y^{(s_y)}f
\bigr)$
holds.
\end{proposition}

\begin{proof}
By Theorem \ref{thm:4.6}, right-hand admissibility is equivalent to
\[
A(y)P(y)\leq f(y)
\text{\quad and \quad}
d^+_{y,P}(t)=o(t^m).
\]
The first condition is a conjunction of homogeneous linear inequalities in $P$ and $f(y)$,
and the second is finite Boolean-linear by Corollary \ref{cor:4.3}. Combining them gives the desired formula. The left-hand case is identical.
\end{proof}

\begin{lemma}[Boolean-linear projection]\label{lem:4.8}
Let $u \in \R^q$, $2 \in \R^r$, and let $\mathcal B(u,w)$ be a finite Boolean combination of homogeneous linear sign conditions. Then
\[
\exists w\in\mathbb R^r
\qquad
\mathcal B(u,w)
\]
is equivalent to a finite Boolean combination of homogeneous linear sign conditions in $u$.
\end{lemma}

\begin{proof}
Put $\mathcal B$ in disjunctive normal form. Projection commutes with finite unions, so it suffices to project a conjunction of linear equalities and strict or weak linear inequalities. Eliminate the coordinates of $w$ one at a time by Fourier--Motzkin elimination, keeping track of strict versus weak bounds. Each step again produces a finite Boolean combination of homogeneous linear sign conditions.
\end{proof}

At $y\in\Sigma_A$, define the $C^m$ compatibility condition
\begin{equation}
\exists
P\in\mathcal P^m(\mathbb R,\mathbb R^M)
\quad
\left[
A(y)P(y)\leq f(y)
\ \wedge\
\mathcal B^-_y
\bigl(
P,J_y^{(s_y)}f
\bigr)
\ \wedge\
\mathcal B^+_y
\bigl(
P,J_y^{(s_y)}f
\bigr)
\right].
\label{eq:exceptional-compatibility}
\end{equation}

\begin{corollary}[Finite exceptional-point compatibility]\label{cor:4.9}
Fix $m\in\{0,1,2\}$. For every $y\in\Sigma_A$, condition \eqref{eq:exceptional-compatibility} is equivalent to a finite Boolean combination of homogeneous linear sign conditions on a finite jet
$J_y^{(s_y)}f$.
All coefficients are determined by $A,M,N,m$.
\end{corollary}

\begin{proof}
Apply Lemma \ref{lem:4.8} to eliminate the finite-dimensional polynomial jet $P\in\mathcal P^m(\mathbb R,\mathbb R^M)$.
\end{proof}

\section{Proof of the positive part of Theorem \ref{thm:main-theorem}} \label{sec:5}

Fix
\[
m\in\{0,1,2\}.
\]
We prove the finite differential characterization of global $C^m$ solvability.

Refine $\Sigma_A$, if necessary, by the finitely many zero-dimensional cells occurring in the Fourier--Motzkin description of pointwise feasibility. Let $\mathcal C$ be the finite semialgebraic partition of $\mathbb R$ consisting of the regular intervals and the singleton cells $\{y: y \in \Sigma_A\}$.

On every regular interval $I$, Corollary \ref{cor:2.2} gives a finite conjunction $\mathcal F_I(f(x))$
of homogeneous linear inequalities such that
\begin{equation}
\mathcal F_I(f(x))
\quad\Longleftrightarrow\quad
K_f(x)\neq\varnothing.
\label{eq:regular-feasibility}
\end{equation}
The coefficients occurring in $\mathcal F_I$ are semialgebraic on $I$ and depend only on the fixed matrix $A$.

At every $y\in\Sigma_A$, Corollary \ref{cor:4.9} gives a finite Boolean formula
\[
\mathcal C_y
\bigl(
J_y^{(s_y)}f
\bigr)
\]
in homogeneous linear sign conditions, equivalent to the $C^m$ exceptional-point compatibility condition.

Set
\[
\bar m
=
\max
\left(
\{m\}
\cup
\{s_y:y\in\Sigma_A\}
\right).
\]
Since $\Sigma_A$ is finite, $\bar m<\infty$. By adding zero coefficients to expressions of lower differential order, all formulas may be regarded as sign formulas in derivatives of $f$ through order $\bar m$.

For a regular cell $I\in\mathcal C$, define
\[
\mathcal S_I
=
\mathcal F_I.
\]
For a singleton cell $\{y\}$, define
\[
\mathcal S_{\{y\}}
=
\mathcal C_y.
\]
These are the cellwise formulas in the positive part of Theorem \ref{thm:main-theorem}.

\subsection{Necessity}

Suppose
$F\in C^m(\mathbb R,\mathbb R^M)$
satisfies $A(x)F(x)\leq f(x)$ for every $x \in \R$.

If $x$ lies in a regular interval $I$, then $F(x) \in K_f(x)$, so $K_f(x)\neq\varnothing$
and \eqref{eq:regular-feasibility} gives
\[
\mathcal F_I(f(x)).
\]
Thus the required cellwise condition holds at every point of every regular interval.

Now let $y\in\Sigma_A$ and put
\[
P_y
=
J_y^mF
\in
\mathcal P^m(\mathbb R,\mathbb R^M).
\]
The restriction of $F$ to the left of $y$ is a feasible one-sided $C^m$ selection with jet $P_y$, and the restriction of $F$ to the right of $y$ is also a feasible one-sided $C^m$ selection with the same jet $P_y$.

By Proposition \ref{pro:4.7}, $P_y$ satisfies both one-sided endpoint formulas. Moreover,
\[
P_y(y)
=
F(y)\in K_f(y),
\]
so $P_y$ witnesses the exceptional-point compatibility condition \eqref{eq:exceptional-compatibility}. Therefore
\[
\mathcal C_y
\bigl(
J_y^{(s_y)}f
\bigr)
\]
holds by Corollary \ref{cor:4.9}.

Thus every cellwise condition in Theorem \ref{thm:main-theorem}  is necessary.

\subsection{Sufficiency}

Conversely, assume that all of the cellwise conditions hold.

For every regular interval $I$, condition
$\mathcal F_I(f(x))$
holds for every $x\in I$. By \eqref{eq:regular-feasibility},
\[
K_f(x)\neq\varnothing
\qquad
(x\in I).
\]
Proposition \ref{pro:2.6} therefore gives
$G_I
\in
C^m(I,\mathbb R^M)$
such that
\[
G_I(x)\in K_f(x)
\qquad
(x\in I).
\]

For every exceptional point
$y\in\Sigma_A$,
the formula
$\mathcal C_y
\bigl(
J_y^{(s_y)}f
\bigr)$
holds. By Corollary \ref{cor:4.9}, there exists a polynomial
$P_y
\in
\mathcal P^m(\mathbb R,\mathbb R^M)$
satisfying the two one-sided endpoint compatibility conditions and
\[
P_y(y)\in K_f(y).
\]
By Theorem \ref{thm:4.6}, there exist feasible one-sided $C^m$ selections
\[
F_y^-
\quad\text{and}\quad
F_y^+
\]
defined on sufficiently small left- and right-hand neighborhoods of $y$, respectively, such that
\begin{equation}
J_y^mF_y^-
=
J_y^mF_y^+
=
P_y.
\label{eq:matching-jets}
\end{equation}

We now patch the regular-interval selections to these endpoint selections.

Let $I=(a,b)$
be a regular interval. When $a$ is finite, use the right-hand endpoint selection $F_a^+$.
When $b$ is finite, use the left-hand endpoint selection
$F_b^-$.
Shrink the endpoint neighborhoods, if necessary, so that the left and right endpoint neighborhoods are disjoint inside $I$.

Choose nonnegative $C^\infty$ cutoff functions $\chi_a$, $\chi_b$
supported in the corresponding endpoint neighborhoods and satisfying
\begin{itemize}
    \item $\chi_a=1$ on a smaller neighborhood of $a$, and
    \item $\chi_b=1$ on a smaller neighborhood of $b$.
\end{itemize}
Modify the supports if necessary so that
\[
\chi_a+\chi_b\leq1,
\]
and put $\chi_0
=
1-\chi_a-\chi_b$.
Thus
\[
\chi_a,\chi_0,\chi_b\geq0,
\qquad
\chi_a+\chi_0+\chi_b=1.
\]

With the evident omissions when an endpoint is infinite, define
\begin{equation}
F_I
=
\chi_aF_a^+
+
\chi_0G_I
+
\chi_bF_b^-.
\label{eq:interval-patching}
\end{equation}
Each term is $C^m$ wherever it appears, hence
\[
F_I\in C^m(I,\mathbb R^M).
\]

At every $x\in I$, all functions appearing in \eqref{eq:interval-patching} with nonzero coefficient belong to the convex polyhedron $K_f(x)$.
Because the coefficients $\chi_a(x),\chi_0(x),\chi_b(x)$
are nonnegative and sum to $1$, convexity gives $F_I(x)\in K_f(x)$.
Therefore,
\[
A(x)F_I(x)\leq f(x)
\text{\qquad for all $x \in I$.}
\]

Moreover, since $\chi_a=1$ near a finite left endpoint $a$ and the other cutoff functions vanish there,
\[
F_I=F_a^+
\]
near $a$. Similarly,
\[
F_I=F_b^-
\]
near a finite right endpoint $b$.

Define the global function $F$ by
\begin{equation*}
    F(y) = \begin{cases}
        F_I(y) &\text{ if $y \in I$ and $I$ is a regular interval}
        \\
        P_y(y) &\text{ if $y \in \Sigma_A$}
    \end{cases} \ .
\end{equation*}

We verify the regularity at the exceptional points. Fix $y\in\Sigma_A$.
On a sufficiently small interval immediately to the left of $y$, the global function $F$ agrees with
$F_y^-$,
while on a sufficiently small interval immediately to the right of $y$, it agrees with
$F_y^+$.
By \eqref{eq:matching-jets},
\[
J_y^mF_y^-
=
J_y^mF_y^+
=
P_y.
\]
Thus the left- and right-hand derivatives of $F$ through order $m$ exist at $y$ and agree, with
\[
F^{(r)}(y)
=
P_y^{(r)}(y),
\qquad
0\leq r\leq m.
\]
Consequently, $F\in C^m(\mathbb R,\mathbb R^M)$.

Feasibility holds on every regular interval by construction. At every exceptional point,
\[
F(y)
=
P_y(y)
\in
K_f(y).
\]
Hence
\[
A(x)F(x)\leq f(x)
\qquad
(x\in\mathbb R).
\]

This proves sufficiency. Therefore, for each
\[
m\in\{0,1,2\},
\]
global $C^m$ solvability is equivalent to the finite collection of cellwise Boolean sign conditions constructed above. This completes the proof of the positive part of the Main Theorem \ref{thm:main-theorem}.
\section{Sharpness of \texorpdfstring{$C^2$}{C2} regularity} \label{sec:6}
\label{sect:sharp C2}

Recall from the introduction that a $C^3(\R) \subset B^3(\R)$ Riesz decomposition associated with a triple of functions $(f_1,f_2,g) \in C^3(\R)^3$ with
\begin{equation*}
    g \leq f_1 + f_2
\end{equation*}
is a pair $(g_1,g_2) \in C^3(\R)^2$ such that
    \begin{equation*}
        g = g_1 + g_2 \text{\quad and \quad} 0 \leq g_i \leq f_i \text{ for $i = 1,2$.}
    \end{equation*}
As before, $B^3$ denotes the subspace of $C^2$ functions $f$ such that $f'''$ exists everywhere and is bounded on compact sets. 

\begin{example}\label{example:NagelRudin}

    We revisit the example given by Nagel and Rudin~\cite{NagelRudin1974}. Let $\set{\psi_k}$ be a sequence of nonnegative univariate smooth functions with pairwise disjoint compact supports in $[0,1]$, such that $\psi_k \equiv 1$ on $(t_k-\delta_k,t_k + \delta_k)$ for some sufficiently small and fast-decreasing $\delta_k$. Choose constants $c_k > 0$ so small such that $\norm{(c_k\psi_k)^{(k)}}_{L^\infty} \leq k^{-2}$. Choose $\eps_k \in (0,\delta_k)$ so that
\begin{equation*}
    \lim_{k\to\infty}\frac{\eps_k}{c_k} = 0.
\end{equation*}
Define
\begin{itemize}
    \item $f_1 := \sum_kf_{1,k}$ where $f_{1,k}(x) := c_k(x-t_k)^2\psi_k(x)$;
    \item $f_2 := \sum_k f_{2,k}$ where $f_{2,k}(x):= c_k\eps_k^2 \psi_k(x)$;
    \item $g(x) := \sum_k g_k$ where $g_k(x) := \frac{1}{2}c_k(x-t_k + \eps_k)^2\psi_k(x)$.
\end{itemize}
It was shown in~\cite{NagelRudin1974} that the triple 
\begin{center}
    \text{$(f_1,f_2,g)$ does \textbf{not} admit a $B^3 $ (and therefore, $ C^3$) Riesz decomposition.}
\end{center}
However, it is worth pointing out that for each $k$, the triple
\begin{center}
    $(f_{1,k}, f_{2,k},g_k)$ admits a $C^3$ (in fact, $C^\infty$) Riesz decomposition
\end{center}
given explicitly by 
\begin{equation*}
    g_{1,k}(x) = \frac{1}{2}c_k\psi_k(x) \frac{s^2(s+\eps_k)^2}{s^2 + \eps_k^2}
    \text{\quad and \quad}
    g_{2,k}(x) = \frac{1}{2}c_k\psi_k(x)\frac{\eps_k^2(s+\eps_k)^2}{s^2 + \eps_k^2}
\end{equation*}
where $s = x - t_k$.
\end{example}

\begin{lemma}\label{lem:lifting}
    Let $(f_1,f_2,g)$ and $f_{i,k}, g_k$ be as in Example~\ref{example:NagelRudin}. For every $x \in \R$, there exists a triple of smooth functions
    \begin{equation*}
        (\tilde f_1^x, \tilde f_2^x, \tilde g^x) \in C^\infty(\R)^3
    \end{equation*}
    admitting a $C^\infty$ Riesz decomposition, such that
    \begin{equation*}
        \jet{\infty}{x}(\tilde f_1^x, \tilde f_2^x, \tilde g^x) \equiv \jet{\infty}{x}(f_1,f_2,g)
    \end{equation*}
    where $\jet{\infty}{x}$ is the formal Taylor series at $x$. Moreover, we can take
    \begin{equation}
        (\tilde f_1^x, \tilde f_2^x, \tilde g^x) = \begin{cases}
            (f_{1,k},f_{2,k}, g_k)
            &\text{ if $x \in \supp(\psi_k)$ for some $k$}
            \\
            (0,0,0) &\text{ otherwise}
        \end{cases} \ . 
        \label{eq:RD triple lift}
    \end{equation}
\end{lemma}

\begin{proof}

    By Example~\ref{example:NagelRudin}, for each $k$, the triple $(f_{1,k}, f_{2,k}, g_k)$ admits a $C^\infty$ Riesz decomposition. The triple $(0,0,0)$ clearly admits a $C^\infty$ Riesz decomposition. It remains to check that the definition \eqref{eq:RD triple lift} gives the same Taylor series.

    We distinguish two cases.

    \textbf{Case A.} Suppose $x \in \supp(\psi_k)$ for some $k$ and fix such $k$. Let $R_i := f_i - f_{i,k}$ and $R_g := g - g_k$. 
    \begin{enumerate}
        \item If $x$ is in the interior of $\supp(\psi_k)$, then $\psi_l \equiv 0$ for all $l \neq k$, the conclusion holds.
        \item If $x$ is not in the interior, there exists a sequence $x_n \to x$ with $\psi_k(x_n) > 0$ such that $\psi_l(x_n)= 0$ for all $l \neq k$.
        This means that
        \begin{equation*}
        \eqindent
            R_1(x_n) = R_2(x_n) = R_g(x_n) = 0
        \end{equation*}
        Since $R_1, R_2, R_g$ are smooth, we have
        \begin{equation*}
             \jet{\infty}{x}R_1 \equiv \jet{\infty}{x}R_2\equiv \jet{\infty}{x}R_g \equiv 0
        \end{equation*}

    \end{enumerate}

    \textbf{Case B.} Suppose $x \notin \bigcup_k \supp(\psi_k)$. We have $f_1(x) = f_2(x) = g(x) = 0$. We then examine the flatness at $x$.
    \begin{itemize}
        \item If $x \notin \overline{\bigcup_k \supp(\psi_k)}$, the conclusion is immediate.
        \item Otherwise, we may argue as in Case A2 and conclude that
        \begin{equation*}
            \jet{\infty}{x}(f_1,f_2,g) \equiv \jet{\infty}{x}(0,0,0).
        \end{equation*}
        Note that $(0,0,0)$ admits the trivial Riesz decomposition.
    \end{itemize}
    This concludes the proof of the lemma. 
\end{proof}

\begin{theorem}\label{thm:C3 sharp}
    Let $A(x) \equiv A = (-1 , 1 , -1 , 1)$ be a constant matrix. There do not exist an integer $\bar m$ and finitely many linear ordinary differential operators
    \begin{equation*}
        L_\nu b(x) = \sum_{i=1}^4 \sum_{r=0}^{\bar m} a_{\nu i r }(x) b_i^{(r)}(x)
        \text{\quad $\nu = 1, \cdots, K$,}
    \end{equation*}
    where each $a_{\nu i r }:\R \to \R$ is an arbitrary function, such that for every $b \in C^\infty(\R,\R^4)$, the existence of $u \in C^3(\R)$ satisfying $Au \leq b$ is equivalent to $L_\nu b(x) \geq 0$ for all $x \in \R$ and $\nu = 1, \cdots, K$. 
\end{theorem}

\begin{proof}
    Let $(f_1,f_2,g)$ be the triple of the functions in Example~\ref{example:NagelRudin} and define
    \begin{equation*}
        b := (0, f_1, f_2 - g, g).
    \end{equation*}
    Note that a function $u$ solves $Au \leq b$ if and only if
    \begin{equation*}
        0 \leq u \leq f_1
        \text{\quad and \quad} 0 \leq g - u \leq f_2,
    \end{equation*}
    if and only if the pair $(u, g-u)$ is a Riesz decomposition associated with $(f_1,f_2,g)$.

    Suppose toward a contradiction that the equivalence holds. Fix an arbitrary $x \in \R$ and $\nu \in \set{1,\cdots, K}$. Let $(\tilde f_1^x, \tilde f_2^x, \tilde g^x)$ be as in Lemma~\ref{lem:lifting} so that
    \begin{equation*}
        \jet{\infty}{x}(\tilde f_1^x, \tilde f_2^x, \tilde g^x) \equiv \jet{\infty}{x}(f_1,f_2,g).
    \end{equation*}
    Setting $\tilde b^x := (0 , \tilde f_1^x , \tilde f_2^x - \tilde g^x , \tilde g^x)$, we also have
    \begin{equation*}
        \jet{\bar m}{x}\tilde b^x = \jet{\bar m}{x}b.
    \end{equation*}
    By Lemma~\ref{lem:lifting}, the triple $(\tilde f_1^x, \tilde f_2^x, \tilde g^x)$ admits a $C^\infty$ (and hence, $C^3$) Riesz decomposition, so the necessity of the equivalence forces
    \begin{equation}
        L_\nu b(x) = L_\nu \tilde b^x(x) \geq 0
        .
        \label{eq:Lnb}
    \end{equation}
    
    Since \eqref{eq:Lnb} holds for all $x \in \R$ and $\nu\in \set{1, \cdots, K}$, the sufficiency of the equivalence would imply that $b$ has a $C^3$ Riesz decomposition, contradicting Example~\ref{example:NagelRudin}.
\end{proof}

As a final remark, Lemma~\ref{lem:lifting} implies a stronger result than Theorem~\ref{thm:C3 sharp}: the obstruction is not pointwise; rather, it arises from accumulations of small oscillations. 

More precisely, for every $x\in \R$, the Nagel-Rudin example $b=(0,f_1,f_2-g,g)$ has the same formal Taylor series as a globally $C^\infty$ solvable datum $\tilde b^x$. As a consequence, $C^{m\geq 3}$ solvability of $Au\leq b$ cannot be characterized by any family of pointwise condition --- finite or infinite --- each depending only on a finite jet of $b$ at the point in question. In particular, this rules out arbitrary Boolean combinations of finite-order linear differential equalities and inequalities imposed pointwise as in~Theorem~\ref{thm:main-theorem},

\bibliographystyle{plain}
\bibliography{bib}

\end{document}